\documentclass[times]{nlaauth_arxiv}%
\usepackage{moreverb}
\usepackage{amsfonts}
\usepackage{amsmath}
\usepackage{amssymb}
\usepackage{graphicx}
\usepackage{subfigure}
\usepackage{url}
\usepackage[pdftex]{thumbpdf}
\usepackage[pdftex,
pdfstartview=FitH,bookmarks=true,bookmarksnumbered=true,bookmarksopen=true,hypertexnames=false,breaklinks=true,
colorlinks=true,  linkcolor=blue,anchorcolor=blue,
citecolor=blue,filecolor=blue,  menucolor=blue]{hyperref}%
\providecommand{\U}[1]{\protect\rule{.1in}{.1in}}
\newtheorem{theorem}{Theorem}
\newtheorem{algorithm}[theorem]{Algorithm}

\newtheorem{lemma}[theorem]{Lemma}

\newtheorem{remark}[theorem]{Remark}

\newcommand\BibTeX{{\rmfamily B\kern-.05em \textsc{i\kern-.025em b}\kern-.08em
T\kern-.1667em\lower.7ex\hbox{E}\kern-.125emX}}

\begin{document}

\title{ }

\title{Hierarchical {S}chur complement preconditioner
for the stochastic {G}alerkin finite element methods
\begin{center}
\small\emph{Dedicated to Professor Ivo Marek on the occasion of his 80th birthday.}
\end{center}
}
\runningheads{B.~Soused\'{i}k et al.}%
{Hierarchical Schur complement preconditioner}
\author{Bed\v{r}ich Soused\'{\i}k\affil{1,2}
Roger G. Ghanem\affil{1}
Eric T. Phipps\affil{3}}
\address{
\affilnum{1}
Department of Aerospace and Mechanical Engineering,
and
Department of Civil and Environmental Engineering,
University of Southern California,
Los Angeles, CA 90089-2531, USA
\break\affilnum{2}
Institute of Thermomechanics,
Academy of Sciences of the Czech Republic,
Dolej\v{s}kova 1402/5, 182~00 Prague~8, Czech Republic
\break\affilnum{3}
Sandia National Laboratories, Albuquerque, NM, USA.}
\cgs{Support from DOE/ASCR is gratefully acknowledged.
B. Soused\'{\i}%
k has been also supported in part by the Grant Agency of the Czech Republic \mbox
{GA \v{C}R 106/08/0403}.
}
\corraddr{B. Soused\'{\i}k,
University of Southern California,
Department of Aerospace and Mechanical Engineering,
Olin Hall (OHE) 430, Los Angeles, CA 90089-2531.
E-mail: sousedik@usc.edu}
\begin{abstract}
Use of the stochastic Galerkin finite element methods leads to large systems
of linear equations obtained by the discretization of tensor product solution
spaces along their spatial and stochastic dimensions. These systems are
typically solved iteratively by a Krylov subspace method. We propose a
preconditioner which takes an advantage of the recursive hierarchy in the
structure of the global matrices. In particular, the matrices posses a
recursive hierarchical two-by-two structure, with one of the submatrices block
diagonal. Each one of the diagonal blocks in this submatrix is closely related
to the deterministic mean-value problem, and the action of its inverse is in the implementation
approximated by inner loops of Krylov iterations. Thus our hierarchical
Schur complement preconditioner
combines, on each level in the approximation of the hierarchical structure of
the global matrix, the idea of Schur complement with loops for a number of
mutually independent inner Krylov iterations, and several matrix-vector
multiplications for the off-diagonal blocks. Neither the global matrix, nor
the matrix of the preconditioner need to be formed explicitly. The ingredients
include only the number of stiffness matrices from the truncated
Karhunen-Lo\`{e}ve expansion and a good preconditioned for the mean-value
deterministic problem.
We provide a condition number bound for a model elliptic problem and
the performance of the method is
illustrated by numerical experiments.
\end{abstract}
\maketitle
%EndExpansion

\section{Introduction}

\label{sec:introduction}A set-up of mathematical models
%used for real-life physical applications
requires
%proper
information about input data. When using partial differential equations
(PDEs), the exact values of boundary and initial conditions along with the
equation coefficients are often not known exactly and instead they need to be
treated with uncertainty. In this study we consider the coefficients as random
parameters. The most straightforward technique of solution is the famous Monte
Carlo method. More advanced techniques, which have became quite
%quickly
popular recently,
%due to their more sophisticated sampling,
include stochastic finite element methods. There are two main variants of
stochastic finite elements: collocation
methods~\cite{Babuska-2010-SCM,Xiu-2005-HCM} and stochastic Galerkin
methods~\cite{Babuska-2005-GFE,Babuska-2005-SEB,Ghanem-1991-SFE}. Both methods
are defined using tensor product spaces for the spatial and stochastic
discretizations. Collocation methods sample the stochastic PDE at a set of
collocation points, which yields a set of mutually independent deterministic
problems. Because one can use existing software to solve this set of problems,
collocation methods are often referred to as non-intrusive. However, the
number of collocation points can be quite prohibitive when high accuracy is
required or when the stochastic problem is described by a large number of
random variables.

On the other hand, the stochastic Galerkin method is intrusive. It uses the
spectral finite element approach to transform a stochastic PDE\ into a coupled
set of deterministic PDEs, and because of this coupling, specialized solvers
are required. The design of iterative solvers for systems of linear algebraic
equations obtained from discretizations by stochastic Galerkin finite element
methods has received significant attention recently. It is well known that
suitable preconditioning can significantly improve convergence of Krylov
subspace iterative methods. Among the most simple, yet quite powerful methods,
belongs the mean-based preconditioner by Powell and
Elman~\cite{Powell-2009-BDP}, cf. also~\cite{Ernst-2009-ESL}. Further
improvements include, e.g., the Kronecker product preconditioner by
Ullmann~\cite{Ullmann-2012-EIS}. We refer to Rosseel and
Vandewalle~\cite{Rosseel-2010-ISS} for a more complete overview and comparison
of various iterative methods and preconditioners, including matrix splitting
and multigrid techniques.
%Most recently, Tipireddy and two of the authors
%investigated~\cite{Tipireddy-2011-CSM} a preconditioner based on
%Jacobi and Gauss-Seidel iterations.
%Interestingly, the preconditioner proposed
%here can be viewed as a generalization of the work of Tipireddy et al.,
%because it could be eventually reduced to their preconditioners.
Also, an interesting approach to solver parallelization
%of the solution of stochastic finite element equations
can be found in the work of Keese and Matthies~\cite{Keese-2005-HPS}.
%However they are concerned mainly with the
%parallelization of the matrix-vector multiply in the Krylov subspace method,
%and with the preconditioner only to a lesser extent and therefore their use of
%the word hierarchical means something else in our present context.

Schur complements are historically well known from substructuring and, in
particular, from the iterative substructuring class of the domain
decomposition methods cf., e.g.,
monographs~\cite{Smith-1996-DD,Toselli-2005-DDM}. However they have also shown
to posses interesting mathematical properties, and they have been studied
independently~\cite{Axelsson-1994-ISM-1Ed,Vassilevski-2008-MBF-1Ed}. The basic
idea is to partition the problem and reorder its matrix representation such
that a direct elimination of a part of the problem becomes straightforward.
%The amount of computational work to obtain the solution is thus significantly reduced.
This reordering can be also performed recursively, which leads to the
recursive Schur complement
methods~\cite{Chen-2005-MPT,Kraus-2012-ASC,Zhang-2000-PSC}. The multilevel
Schur complement preconditioning in multigrid framework can be, to the best of
our knowledge, traced back to Axelsson and
Vassilevski~\cite{Axelsson-1989-AMP,Axelsson-1990-AMP}. The Algebraic
Recursive Multilevel Solver (ARMS) by Saad and Suchomel~\cite{Saad-2002-ARMS}
and its parallel version (pARMS) by Li et al.~\cite{Li-2003-pARMS} use
variants of incomplete LU\ decompositions, and they are also closely related
to the Hierarchical Iterative Parallel Solver (HIPS) by Gaidamour and
H\'{e}non~\cite{Gaidamour-2008-PDS}.
%Extension to applications \cite{Saad-2002-AAR,Sosonkina-2004-UPA},\ see also
%\cite{Saad-2007-SCP}. General theory by Zhongzhi and Axelsson~\cite{Zhongzhi-1999-UFC}.
We also note that a remarkable idea for preconditioning non-symmetric systems
using an approximate Schur complement has been proposed by Murphy, Golub and
Wathen~\cite{Murphy-2000-NPI}.
%However, here we restrict our considerations to the symmetric case.

In this paper, we propose a symmetric preconditioner which takes advantage of
the recursive hierarchy in the structure of the global system matrices. This
structure is obtained directly from the stochastic formulation. In particular,
the matrices posses a recursive hierarchical two-by-two structure,
cf.~\cite{Ghanem-1996-NSS,Pellissetti-2000-ISS}, where one of the submatrices
is block diagonal and therefore its inverse can be computed by inverting each
of the blocks independently. Moreover, each of the diagonal blocks is closely
related to the deterministic mean-value problem. In fact, the diagonal blocks
are obtained simply by rescaling the mean-value matrix in the case of linear
Karhunen-Lo\`{e}ve expansion. So, assuming that we have a good preconditioner
for the mean available, each block can be solved iteratively by an inner loop
of Krylov iterations. Doing so, our hierarchical Schur complement
preconditioner becomes variable because it combines, on each level in the
approximation of the hierarchical structure of the global matrix, the idea of
the Schur complement with loops for a number of mutually independent inner
Krylov iterations, and several matrix-vector multiplications for the
off-diagonal blocks. Due to variable preconditioning one has to make a careful
choice of Krylov subspace methods, and their variants such as flexible
conjugate gradients~\cite{Notay-2000-FCG}, FGMRES~\cite{Saad-1993-FIP}, or
GMRESR~\cite{vanderVorst-1994-GFN} are preferred. However, in our numerical
experiments, we have obtained the same convergence with the flexible and the
standard versions of conjugate gradients. It is important to note that neither
the global matrix, nor the preconditioner need to be formed explicitly, and we
can use the so called MAT-VEC operations from~\cite{Pellissetti-2000-ISS} in
both matrix-vector multiplications: by a global system matrix in the loop of
outer iterations and in the action of the preconditioner. The ingredients of
our method thus include only the number of stiffness matrices from the
truncated Karhunen-Lo\`{e}ve expansion and a good preconditioner for the
mean-value deterministic problem. Therefore the method can be regarded as
minimally intrusive because it can be built as a wrapper around an existing
solver for the corresponding mean-value problem.
%, which will be called from
%within the preconditioner number of times.
%The method also takes maximum advantage of the sparsity pattern of
%the global matrix and, because the wrapper itself performs only sparse
%matrix-vector multiplies, it can programmed in parallel in a relatively
%straightforward manner.
Nevertheless in this contribution we neither address the parallelization nor
the choice of the preconditioner for the mean-value problem. These two topics
would not change the convergence in terms of outer iterations, and they will
be studied elsewhere.
%The preconditioned improves the convergence rate quite
%significantly which is illustrated by a number of numerical experiments.

The paper is organized as follows. In Section~\ref{sec:model} we introduce the
model problem, in Section~\ref{sec:structure} we discuss the structure of the
stochastic matrices, in Section~\ref{sec:schur} we formulate the hierarchical
Schur complement preconditioner and provide a condition number bound under
suitable assumptions, in Section~\ref{sec:variants} we outline possible
variants of the method and provide details of our implementation, and finally,
in Section~\ref{sec:numerical} we illustrate the performance of the algorithm
by numerical experiments, and in Section~\ref{sec:summary} we provide a short
summary and a conclusion of the work presented in this paper.

\section{Model problem and its discretization}

\label{sec:model}Let $D$ be a domain in $%
%TCIMACRO{\U{211d} }%
%BeginExpansion
\mathbb{R}
%EndExpansion
^{d}$, $d=2$, and let $\left(  \Omega,\mathcal{F},\mu\right)  $ be a complete
%\footnote{In the measure theory sense, i.e. every negligible set is measurable.}
probability space, where $\Omega$ is the sample space, $\mathcal{F}$ is the
$\sigma-$algebra generated by $\Omega$ and $\mu:\mathcal{F}\rightarrow\left[
0,1\right]  $ is the probability measure. We are interested in a solution of
the following elliptic boundary value problem: find a random function
$u\left(  x,\omega\right)  :\overline{D}\times\Omega\rightarrow%
%TCIMACRO{\U{211d} }%
%BeginExpansion
\mathbb{R}
%EndExpansion
$ which almost surely (a.s.) satisfies the equation%
\begin{align}
-\nabla\cdot\left(  k\left(  x,\omega\right)  \,\nabla u\left(  x,\omega
\right)  \right)   &  =f\left(  x\right)  \qquad\text{in}\ D\times
\Omega,\label{eq:model-1}\\
u\left(  x,\omega\right)   &  =0\qquad\text{on}\ \partial D\times
\Omega,\label{eq:model-2}%
\end{align}
where $f\in L^{2}\left(  D\right)  $, and $k\left(  x,\omega\right)  $ is a
random scalar field with a probability density function$~d\mu\left(
\omega\right)  .$ We note that the gradient symbol $\nabla$ denotes the
differentiation with respect to the spatial variables. Also, we will assume
that there exist two constants $0<k_{\min}\leq k_{\max}$\ such that%
\[
\mu\left(  \omega\in\Omega:k_{\min}\leq k\left(  x,\omega\right)  \leq
k_{\max}\;\forall x\in\overline{D}\right)  =1.
\]
In the weak formulation of problem (\ref{eq:model-1})-(\ref{eq:model-2}),
we would like to solve%
\begin{equation}
u\in U:\ a\left(  u,v\right)  =\left\langle f,v\right\rangle ,\qquad\forall
v\in U.\label{eq:model-variational}%
\end{equation}
Here $f\in U^{\prime}$ with $U^{\prime}$ denoting the dual of $U$ and
$\left\langle \cdot,\cdot\right\rangle $ the duality pairing. The space $U$
and its norm are defined, using a tensor product and expectation $\mathbb{E}$
with respect to the measure$~\mu$, as
\[
U=H_{0}^{1}\left(  D\right)  \otimes L_{\mu}^{2}\left(  \Omega\right)
,\qquad\left\Vert u\right\Vert _{U}=\sqrt{\mathbb{E}\left[  \int_{D}\left\vert
\nabla u\right\vert ^{2}dx\right]  }.
\]
The bilinear form $a$ and right-hand side are
\[
a\left(  u,v\right)  =\mathbb{E}\left[  \int_{D}k\left(  x,\omega\right)
\,\nabla u\cdot\nabla v\,dx\right]  ,\qquad\left\langle f,v\right\rangle
=\mathbb{E}\left[  \int_{D}f\,v\,dx\right]  .
\]
Next, let us define the stochastic operator $K_{\omega}:U\rightarrow
U^{\prime}$
%associated with the bilinear form $a$
by
\begin{equation}
a\left(  u,v\right)  =\left\langle K_{\omega}u,v\right\rangle ,\qquad\forall
u,v\in U.\label{eq:a-K}%
\end{equation}
So the problem~(\ref{eq:model-variational}) can be now equivalently written as
the stochastic operator equation
\begin{equation}
\left\langle K_{\omega}\,u,v\right\rangle =\left\langle f,v\right\rangle
,\qquad\forall v\in U.\label{eq:model-operator}%
\end{equation}
The operator $K_{\omega}$\ is stochastic via the random parameter$~k\left(
x,\omega\right)  $. Assuming that its covariance function $C\left(
x_{1},x_{2}\right)  $ is known, we will further assume that it has the linear
Karhunen-Lo\`{e}ve (KL) expansion truncated after $N$ terms as%
\begin{equation}
k\left(  x,\omega\right)  =\sum_{i=0}^{N}k_{i}\left(  x\right)  \xi_{i}\left(
\omega\right)  ,\quad\xi_{0}=1,\quad\xi_{i}\sim\text{U}\left[  0,1\right]
\quad i=1,\dots,N,\label{eq:KL-expansion}%
\end{equation}
such that $\xi_{i}\left(  \omega\right)  $, $i>0$\ are identically
distributed, independent random variables. Here $k_{0}$ is the mean of the
random field, and $k_{i}\left(  x\right)  =\sqrt{\lambda_{i}}v_{i}\left(
x\right)  $ where $\left(  \lambda_{i},v_{i}\left(  x\right)  \right)
_{i\geq1}$ are the solutions of the integral eigenvalue problem
\begin{equation}
\int_{D}C\left(  x_{1},x_{2}\right)  v_{i}\left(  x_{2}\right)  dx_{2}%
=\lambda_{i}v_{i}\left(  x_{1}\right)  ,\label{eq:int-eigproblem}%
\end{equation}
see~\cite{Ghanem-1991-SFE}\ for details. For the numerical experiments in this
paper, we made a specific choice
\begin{equation}
C\left(  x_{1},x_{2}\right)  =\sigma^{2}\exp\left(  -\left\Vert x_{1}%
-x_{2}\right\Vert _{1}/L\right)  ,\label{eq:cov}%
\end{equation}
with $\sigma^{2}$ denoting the variance, and $L$ the correlation length of the
random variables$~\xi_{i}\left(  \omega\right)  $. Efficient computational
methods for solution of the eigenvalue problem (\ref{eq:int-eigproblem}) are
described, e.g., in~\cite{Schwab-2006-KLA}.

Using the KL expansion of $k$ in the definition of the operator $K_{\omega}$
in~(\ref{eq:a-K}), we obtain%
\begin{equation}
\left\langle K_{\omega}u,v\right\rangle =\left\langle \sum_{i=0}^{N}\xi
_{i}\left(  \omega\right)  k_{i}\left(  x\right)  u\left(  x,\omega\right)
,v\left(  x,\omega\right)  \right\rangle . \label{eq:gPC-K}%
\end{equation}

\begin{remark}
\label{rem:K-L} More generally than (\ref{eq:KL-expansion}), we can consider
the generalized polynomial chaos (gPC) expansion of $k$ as%
\[
%begin{equation}
k\left(  x,\omega\right)  =\sum_{i=0}^{M^{\prime}}k_{i}\left(  x\right)
\psi_{i}\left(  \mathbb{\xi(\omega)}\right)  .
%\label{eq:K-L}
\]
%end{equation}
%Clearly the gPC expansion can be written as the truncated KL
%expansion~(\ref{eq:KL-expansion}) using Legendre basis of polynomials~$\Lambda
%_{0},\Lambda_{1},\dots,$ and setting
%\begin{align*}
%k_{0} &  =\mu,\qquad\psi_{0}=\Lambda_{0}=1,\\
%k_{i} &  =\sigma,\qquad\psi_{i}=\Lambda_{1}\left(  \xi_{i}\right)  =\xi
%_{i},\qquad1\leq i\leq N,\\
%k_{i} &  =0,\qquad i>N.
%\end{align*}
In both cases, we write
\[
%begin{equation}
k\left(  x,\omega\right)  =\sum_{i=0}^{L}k_{i}\left(  x\right)  \psi
_{i}\left(  \mathbb{\xi(\omega)}\right)  ,
%\label{eq:K-L}
\]
for $L=N$ in the KL case and $L=M^{\prime}$ in the gPC case.
\end{remark}

We will consider discrete approximations to the solution
to~\eqref{eq:model-operator} given by finite element discretizations of
$H_{0}^{1}\left(  D\right)  $ and generalized polynomial chaos (gPC)
discretizations of $L_{\mu}^{2}\left(  \Omega\right)  $, namely
\begin{equation}
u=\sum_{i=1}^{N_{dof}}\sum_{j=0}^{M}u_{ij}\phi_{i}(x)\psi_{j}\left(  \xi
_{0},\dots,\xi_{N}\right)  , \label{eq:gPC-u}%
\end{equation}
where $\left\{  \phi_{i}(x) \right\}  _{i=1}^{N_{dof}}$ are suitable finite
element basis functions, the gPC\ basis $\left\{  \psi_{j}(\xi) \right\}
_{j=0}^{M}$ is obtained as the the tensor product of Legendre polynomials of
total order at most $P$ and $\xi= \left(  \xi_{0},\dots,\xi_{N}\right)  $. The
choice of Legendre polynomials is motivated by the fact that these are
orthogonal with respect to the probability measure associated with the uniform
random variables $\xi_{0},\dots,\xi_{N}$. The total number of gPC polynomials
is thus $M+1=\frac{(N+P)!}{N!P!}$, cf. also~\cite[p. 87]{Ghanem-1991-SFE}.

Substituting the expansions (\ref{eq:gPC-K}) and (\ref{eq:gPC-u}) into
(\ref{eq:model-operator}) yields a deterministic linear system of equations
\begin{equation}
\sum_{j=0}^{M}\sum_{i=0}^{L}c_{ijk}K_{i}u_{j}=f_{k},\qquad k=0,\dots
,M,\label{eq:global-system}%
\end{equation}
where $(f_{k})_{l}=\mathbb{E}\left[  \int_{D}f\left(  x\right)  \phi
_{l}\left(  x\right)  \psi_{k}\,dx\right]  $, $\left(  K_{i}\right)
_{lm}=\int_{D}k_{i}(x)\phi_{l}(x)\phi_{m}(x)\,dx$, and the coefficients
$c_{ijk}=\mathbb{E}\left[  \psi_{i}\psi_{j}\psi_{k}\right]  $. Each one of the
blocks~$K_{i}$ is thus a deterministic\ stiffness matrix given by$~k_{i}%
\left(  x\right)  $, cf.~(\ref{eq:gPC-K}), of size $\left(  N_{dof}\times
N_{dof}\right)  $, where $N_{dof}$ is the number of spatial degrees of
freedom. The system~(\ref{eq:global-system}) is then given by a global matrix
of size $\left(  \left(  M+1\right)  N_{dof}\times\left(  M+1\right)
N_{dof}\right)  $, consisting of $N_{dof}\times N_{dof}$\ blocks $K^{\left(
j,k\right)  }$, and it can be written as
\begin{equation}
\left[
\begin{array}
[c]{ccccc}%
K^{\left(  0,0\right)  } & K^{\left(  0,1\right)  } & \cdots &  & K^{\left(
0,M\right)  }\\
& \ddots &  &  & \\
\vdots &  & K^{\left(  k,k\right)  } &  & \vdots\\
&  &  & \ddots & \\
K^{\left(  M,0\right)  } & K^{\left(  M,1\right)  } & \cdots &  & K^{\left(
M,M\right)  }%
\end{array}
\right]  \left[
\begin{array}
[c]{c}%
u_{0}\\
\vdots\\
u_{k}\\
\vdots\\
u_{M}%
\end{array}
\right]  =\left[
\begin{array}
[c]{c}%
f_{0}\\
\vdots\\
f_{k}\\
\vdots\\
f_{M}%
\end{array}
\right]  ,\label{eq:global-matrix}%
\end{equation}
where each of the blocks $K^{\left(  j,k\right)  }$ is in the KL\ case
obtained as
\begin{equation}
K^{\left(  j,k\right)  }=\sum_{i=0}^{N}c_{ijk}K_{i}%
.\label{eq:global-matrix-block}%
\end{equation}

\begin{remark}
With an iterative solution of~(\ref{eq:global-matrix}) in mind, one needs to
store only the constants$~c_{ijk}$, the blocks $K_{i}$ and use the formula
(\ref{eq:global-matrix-block}) for matrix-vector multiplication, see MAT-VEC
operations in~\cite{Pellissetti-2000-ISS}.
\end{remark}

It is important to note that the first diagonal block is obtained by the
$0-$th order polynomial chaos expansion and therefore it corresponds to the
deterministic problem obtained using the mean value of the coefficient$~k$, in
particular
\[
K^{\left(  0,0\right)  }=K_{0}.
\]

The sparsity structure of the matrix in~(\ref{eq:global-matrix}) will in
general depend on the type of the gPC\ polynomial basis, on the number of
terms retained in the expansions (\ref{eq:gPC-K}) and (\ref{eq:gPC-u}), and
also on the number of stochastic dimensions. Nevertheless, due to the
orthogonality of the gPC\ basis functions, the constants$~c_{ijk}$\ will
vanish for many combinations of the indices~$i$, $j$, and $k.$ The block
sparsity structure of the global stochastic Galerkin matrix~in
(\ref{eq:global-matrix}), with the blocks given by
(\ref{eq:global-matrix-block}), will depend on a matrix$~c_{P}$ with entries
$c^{\left(  j,k\right)  }=\sum_{i=0}^{N}c_{ijk}$, where$~j,k=0,\dots,M$. The
typical structure of$~c_{P}$ is illustrated by Figure~\ref{fig:NP}. Looking
carefully at the figures, we can observe a block hierarchical structure of the
matrices. In the next section, we will study this structure in somewhat more detail.

%\begin{figure}[tbp]
%\begin{center}
%\includegraphics[width=6cm]{graphics/N4P7.pdf}
%\end{center}
%\caption{Example of a configuration of planar cubes test problem with 36
%subdomains and $H/h=8$, red dots represent corners.}
%\label{fig:N4P7}
%\end{figure}

\begin{figure}[th]
\centering
\subfigure[$N=4$, $P=4$ gives $350$ blocks]{
\includegraphics[width=6.5cm]{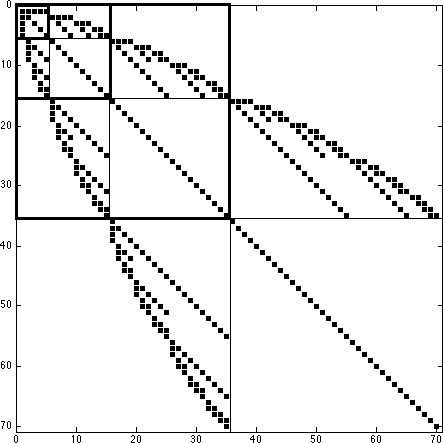} \label{fig:N4P4}
} \subfigure[$N=4$, $P=7$ gives $2010$ blocks]{
\includegraphics[width=6.5cm]{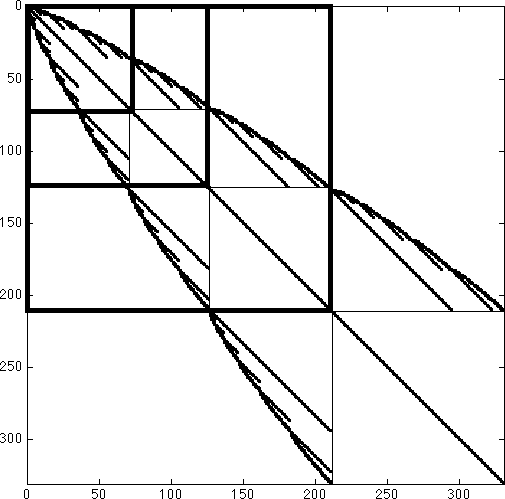}
\label{fig:N4P7}
}\caption{Hierarchical structure of the matrix $c_{P}$ which determines the
block sparsity of the global stochastic Galerkin matrix with $N=4$ stochastic
dimensions using~\subref{fig:N4P4} $P=4$, or~\subref{fig:N4P7} $P=7$ order of
polynomial expansion. The sub-blocks correspond to the polynomials of order
\subref{fig:N4P4} $P=1,2,3$ and~\subref{fig:N4P7} $P=4,5,6$.}%
\label{fig:NP}%
\end{figure}

\section{Structure of the model matrices}

\label{sec:structure}Let us begin by an illustration. Figure~\ref{fig:N4P4}
shows the structure of the stochastic Galerkin matrix based on the fourth
order polynomial chaos expansion in four stochastic dimensions. The schematic
matrix in the picture is$~c_{P}$ (here $P=4$), so in the global stochastic
Galerkin matrix as it is written in eq.~(\ref{eq:global-matrix}) each tile
corresponds to a block of a stiffness matrix with the same sparsity pattern as
the original finite element problem. Now, let us denote the corresponding
global Galerkin matrix by$~A_{4}$ , and by $A_{3}$, $B_{4}$, $C_{4}$ and
$D_{4}$ its four submatrices, cf.~(\ref{eq:global-matrix-hierarchy}). We see
that$~D_{4}$ is block diagonal and the structure of$~A_{3}$ resembles the
structure of$~A_{4}$ and this hierarchy is repeated all the way to the
$1\times1$ block$~A_{0}$ and a block diagonal matrix$~D_{1}$. The number in
the subscript indicates that the entries in the block correspond to the
polynomial expansion in the case of (a) $A_{3}$ of order three or less, and
(b) $B_{3},\ C_{3}\ $and$\ D_{3}$ of order three. Clearly, the sparsity and
hierarchical structure follows from orthogonality of the polynomials as was
pointed out in~\cite{Pellissetti-2000-ISS}. More specifically, let us consider
a two-by-two block structure of a~(square) coefficient matrix $c_{P}$ with
dimensions $\frac{\left(  N+P\right)  !}{N!P!}$ as
\[
c_{P}=\left[
\begin{array}
[c]{cc}%
c_{P-1} & b_{P}^{T}\\
b_{P} & d_{P}%
\end{array}
\right]  ,
\]
where $c_{P-1}$ is the first principal\ submatrix with dimensions
$\frac{\left(  N+P-1\right)  !}{N!\left(  P-1\right)  !}$, and the remaining
blocks are defined accordingly. Generally, let us consider a recursive
hierarchy in the spliting of $c_{P}$ as
\[
c_{\ell}=\left[
\begin{array}
[c]{cc}%
c_{\ell-1} & b_{\ell}^{T}\\
b_{\ell} & d_{\ell}%
\end{array}
\right]  ,\qquad\ell=P,\dots,1,
\]
where the dimensions of$~c_{\ell}$\ are given by $\frac{\left(  N+\ell\right)
!}{N!\ell!}$, the dimensions of the first principal submatrices$~c_{\ell-1}$
are given by $\frac{\left(  N+\ell-1\right)  !}{N!\left(  \ell-1\right)  !}$,
and the remaining blocks are defined accordingly. We note that even though the
matrices$~c_{\ell}$ are symmetric, the stochastic Galerkin matrix will be
symmetric only if each one of the matrices$~K_{i}$ is itself symmetric. We
refer, e.g., to~\cite{Rosseel-2010-ISS,Ernst-2010-SGM} for further details and
discussion, and state here only the essential observation for our approach:

\begin{lemma}
[{\cite[Corollary 2.6]{Rosseel-2010-ISS}}]The block $d_{\ell}$ is a diagonal matrix for all $\ell=1,\dots,P$.
\end{lemma}

%\begin{proof}
%Note that $\psi_{i}\psi_{j}$ is a polynomial of order $\alpha_{i}+\alpha_{j}$
%and it can be expressed as $\psi_{i}\psi_{j}=\sum_{k}c_{ijk}\psi_{k}$, where
%the summation includes polynomials of orders up to $\alpha_{i}+\alpha_{j}$.
%When projecting this last expression on $\psi_{k}$, we get $\mathbb{E}\left[
%\psi_{i}\psi_{j}\psi_{k}\right]  =c_{ijk}$ which is zero whenever $\alpha
%_{k}>\alpha_{i}\alpha_{j}.$This is in fact satisfied for any system of
%orthogonal polynomials.
%\end{proof}

The global problem~(\ref{eq:global-matrix}) can be equivalently written as%
\begin{equation}
A_{P}u_{P}=f_{P}, \label{eq:global-problem}%
\end{equation}
with the matrix $A_{P}$ having a hierarchical structure
\begin{equation}
A_{\ell}=\left[
\begin{array}
[c]{cc}%
A_{\ell-1} & B_{\ell}\\
C_{\ell} & D_{\ell}%
\end{array}
\right]  ,\qquad\ell=P,\dots,1, \label{eq:global-matrix-hierarchy}%
\end{equation}
where the subscript$~\ell$ stands for the blocks obtained by an approximation
by the $\ell-$th degree stochastic polynomial (or lower), and all of the
blocks $D_{\ell}$ are block diagonal. In particular the smallest case is given
by the finite element approximation with the mean values of the coefficients,
and therefore the \emph{mean-value problem} is
\begin{equation}
A_{0}u_{0}=f_{0}, \label{eq:mean-value-problem}%
\end{equation}
and in particular $A_{0}=K_{0}$. In this paper, we will assume that the
inverse of $A_{0}$ is known, or at least that we have a good preconditioner
$M_{0}$\ readily available.

\begin{remark}
Clearly, if all of the matrices$~K_{i}$ are symmetric, the global
matrix$~A_{P}$ and all of its submatrices$~A_{\ell}$\ will be symmetric as
well, i.e.,
\[
A_{\ell}=\left[
\begin{array}
[c]{cc}%
A_{\ell-1} & B_{\ell}\\
B_{\ell}^{T} & D_{\ell}%
\end{array}
\right]  ,\qquad\ell=P,\dots,1.
\]
However, for the sake of generality, we will use the non-symmetric
notation~(\ref{eq:global-matrix-hierarchy}). We note that a question under
what conditions is the global problem positive definite is far more delicate,
in general depends on the type of the polynomial expansion and also on the
choice of the covariance function.
\end{remark}

In the next section we introduce our preconditioner, taking advantage of the
hierarchical structure and of the fact that the matrices $D_{\ell}$, where
$\ell=P,\dots,1$, are block diagonal.

\section{Schur complement preconditioner}

\label{sec:schur}Let us find an inverse of a general block matrix given as%
\begin{equation}
\left[
\begin{array}
[c]{cc}%
A & B\\
C & D
\end{array}
\right]  , \label{eq:block-matrix}%
\end{equation}
assuming that we can easily compute the inverse of $D$. By block
LU\ decomposition, we can derive
\begin{equation}
\left[
\begin{array}
[c]{cc}%
A & B\\
C & D
\end{array}
\right]  =\left[
\begin{array}
[c]{cc}%
I_{A} & BD^{-1}\\
0 & I_{D}%
\end{array}
\right]  \left[
\begin{array}
[c]{cc}%
S & 0\\
0 & D
\end{array}
\right]  \left[
\begin{array}
[c]{cc}%
I_{A} & 0\\
D^{-1}C & I_{D}%
\end{array}
\right]  , \label{eq:block-LU}%
\end{equation}
where $S=A-BD^{-1}C$ is the Schur complement of$~D$ in~(\ref{eq:block-matrix}%
). Inverting the three blocks, we obtain
\begin{equation}
\left[
\begin{array}
[c]{cc}%
A & B\\
C & D
\end{array}
\right]  ^{-1}=\left[
\begin{array}
[c]{cc}%
I_{A} & 0\\
-D^{-1}C & I_{D}%
\end{array}
\right]  \left[
\begin{array}
[c]{cc}%
S^{-1} & 0\\
0 & D^{-1}%
\end{array}
\right]  \left[
\begin{array}
[c]{cc}%
I_{A} & -BD^{-1}\\
0 & I_{D}%
\end{array}
\right]  . \label{eq:block-inverse}%
\end{equation}

%\begin{remark}
%In iterative substructuring, the "easy-to-invert" block is usually written in
%place of the block $A$ at the position $\left(  1,1\right)  $, cf.,
%e.g.~\cite[Chapter 4]{Toselli-2005-DDM} or~\cite{Li-2006-FBB}.
%\end{remark}
The hierarchical Schur complement preconditioner is based on the block
inverse~(\ref{eq:block-inverse}). In the action of the preconditioner,
application of the three blocks on the right-hand side
of~(\ref{eq:block-inverse}) will be called (in the order in which they are
performed) as \emph{pre-correction}, \emph{correction} and
\emph{post-correction}.

%Note
%that
%%
%%if the problem is symmetric, i.e., if $C=B^{T}$,
%this multiplication requires only two matrix vector multiplies by $D^{-1}$ and
%one by $S^{-1}$; one might save the product of the first part of a vector with
%$D^{-1}$ and reuse it in the (post-)correction.

So, in the action of the preconditioner we would like to approximate
problem~(\ref{eq:global-problem}) which with respect
to~(\ref{eq:global-matrix-hierarchy}) can be written as%
\begin{equation}
\left[
\begin{array}
[c]{cc}%
A_{P-1} & B_{P}\\
C_{P} & D_{P}%
\end{array}
\right]  \left[
\begin{array}
[c]{c}%
u_{P}^{P-1}\\
u_{P}^{P}%
\end{array}
\right]  =\left[
\begin{array}
[c]{c}%
f_{P}^{P-1}\\
f_{P}^{P}%
\end{array}
\right]  . \label{eq:global-matrix-hierarchy-2}%
\end{equation}
The matrix inverse can be with respect to (\ref{eq:block-inverse}) written as%
\[
\left[
\begin{array}
[c]{cc}%
A_{P-1} & B_{P}\\
C_{P} & D_{P}%
\end{array}
\right]  ^{-1}=\left[
\begin{array}
[c]{cc}%
I_{A} & 0\\
-D_{P}^{-1}C_{P} & I_{D}%
\end{array}
\right]  \left[
\begin{array}
[c]{cc}%
S_{P-1}^{-1} & 0\\
0 & D_{P}^{-1}%
\end{array}
\right]  \left[
\begin{array}
[c]{cc}%
I_{A} & -B_{P}D_{P}^{-1}\\
0 & I_{D}%
\end{array}
\right]  ,
\]
where%
\[
S_{P-1}=A_{P-1}-B_{P}D_{P}^{-1}C_{P}.
\]
Because computing (and inverting) the Schur complement $S_{P-1}$ explicitly is
computationally prohibitive, we suggest to replace the inverse of $S_{P-1}$ by
the inverse of~$A_{P-1}$. Since $A_{P-1}$ has the hierarchical structure as
described by~(\ref{eq:global-matrix-hierarchy}), i.e.,
\[
A_{P-1}=\left[
\begin{array}
[c]{cc}%
A_{P-2} & B_{P-1}\\
C_{P-1} & D_{P-1}%
\end{array}
\right]  ,
\]
we can approximate its inverse again using the idea of (\ref{eq:block-inverse}%
) and so on.
%http://www.cs.umd.edu/class/fall2002/cmsc214/Tutorial/recursion2.html
Eventually, we arrive at the Schur complement of the mean-value problem$~S_{0}%
$\ which we replace by$~A_{0}$.
%using the mean-value preconditioner $M_{0}$.
Thus the action of this hierarchical\ preconditioner$~M_{P}$ consists of a
number of pre-correction steps performed on the levels$~\ell=P,\dots,1$,
solving the \textquotedblleft mean-value\textquotedblright\ problem
with$~A_{0}$ on the lowest level, and performing a number of the
post-processing steps sweeping up the levels. We now formulate the
preconditioner for the iterative solution of the global
problem~(\ref{eq:global-problem}) more concisely as:

\begin{algorithm}
[Hierarchical Schur complement preconditioner]%
\label{alg:hierarchical-schur-prec} The preconditioner $M_{P}:r_{P}\longmapsto
u_{P}$\ is defined as follows:

\noindent\textbf{for} $\ell=P,\ldots1$\textbf{,}

\begin{description}
\item split the residual, based on the hierarchical structure of matrices, as
\[
r_{\ell}=\left[
\begin{array}
[c]{c}%
r_{\ell}^{\ell-1}\\
r_{\ell}^{\ell}%
\end{array}
\right]  ,
\]

\item compute the pre-correction as
\[
g_{\ell-1}=r_{\ell}^{\ell-1}-B_{\ell}D_{\ell}^{-1}r_{\ell}^{\ell}.
\]

\item If $\ell>1$, set
\[
r_{\ell-1}=g_{\ell-1}.
\]

\item Else (if $\ell=1)$, solve the system $A_{0}u_{0}=g_{0}$.
%so that $\left\Vert g_{0}-S_{0}u_{0}\right\Vert <tol.$

\end{description}

\noindent\textbf{end}

\noindent\textbf{for} $\ell=1,\ldots P$\textbf{,}

\begin{description}
\item compute the post-correction, i.e., set $u_{\ell}^{\ell-1}=u_{\ell-1}$,
solve
\[
u_{\ell}^{\ell}=D_{\ell}^{-1}\left(  r_{\ell}^{\ell}-C_{\ell}u_{\ell}^{\ell
-1}\right)  ,
\]

\item and concatenate%
\[
u_{\ell}=\left[
\begin{array}
[c]{c}%
u_{\ell}^{\ell-1}\\
u_{\ell}^{\ell}%
\end{array}
\right]  .
\]

\item If $\ell<P,$ set $u_{\ell+1}^{\ell}=u_{\ell}$.
\end{description}

\noindent\textbf{end}
\end{algorithm}

We will now restrict our considerations to the case when all of the matrices~
%$~K_{i},$ $i=1,\dots,L$\ and
$A_{\ell}$, $\ell=P,\dots,1$\ are symmetric, positive definite. In this case,
the decomposition~(\ref{eq:block-LU}) can be written for all levels $\ell$ as
\[
A_{\ell}=\left[
\begin{array}
[c]{cc}%
A_{\ell-1} & B_{\ell}\\
B_{\ell}^{T} & D_{\ell}%
\end{array}
\right]  =\left[
\begin{array}
[c]{cc}%
I_{A} & B_{\ell}D_{\ell}^{-1}\\
0 & I_{D}%
\end{array}
\right]  \left[
\begin{array}
[c]{cc}%
S_{\ell-1} & 0\\
0 & D_{\ell}%
\end{array}
\right]  \left[
\begin{array}
[c]{cc}%
I_{A} & 0\\
D_{\ell}^{-1}B_{\ell}^{T} & I_{D}%
\end{array}
\right]  .
\]
Because all of the matrices $A_{\ell}$, $\ell=P,\dots,1$ are positive
definite, the above becomes a set of congruence transformations and by the
Sylvester law of inertia, all of the Schur complements$~S_{\ell}$,
$\ell=P-1,\dots,0$ are also symmetric positive definite. Thus, we can
establish for appropriate vectors~$u$ the next set of
%energy norm
inequalities,
\begin{equation}
c_{\ell,1}\left\Vert u\right\Vert _{A_{\ell}}^{2}\leq\left\Vert u\right\Vert
_{S_{\ell}}^{2}\leq c_{\ell,2}\left\Vert u\right\Vert _{A_{\ell}}^{2}%
,\qquad\ell=0,\dots,P-1, \label{eq:AS-ineq}%
\end{equation}
where $\left\Vert u\right\Vert _{A}^{2}=u^{T}Au$ denotes the energy norm, and
use it in the following:

\begin{theorem}
\label{thm:cond}For the symmetric, positive definite matrix$~A_{P}$ the
preconditioner$~M_{P}$ defined by Algorithm~\ref{alg:hierarchical-schur-prec}%
\ is also positive definite, and the condition number $\kappa$\ of the
preconditioned system is bounded by
\[
\kappa=\frac{\lambda_{\max}\left(  M_{P}A_{P}\right)  }{\lambda_{\min}\left(
M_{P}A_{P}\right)  }\leq C,\qquad\text{where }C=\Pi_{\ell=0}^{P-1}%
\frac{c_{\ell,2}}{c_{\ell,1}}.
\]

\end{theorem}

\begin{proof}
The bound follows directly from the sequential replacement of the Schur
complement operators~$S_{\ell}$ by the hierarchical matrices~$A_{\ell}$\ in
Algorithm~\ref{alg:hierarchical-schur-prec}, and the bounds in the equivalence~(\ref{eq:AS-ineq}).
\end{proof}

Hence, the convergence rate can be established from the spectral
equivalence~(\ref{eq:AS-ineq}).

\begin{remark}
Despite the multiplicative growth of the condition number bound as predicted
by Theorem~\ref{thm:cond} from our numerical experiments
(Table~\ref{tab:Legendre-var-P}) it appears that, at least in the case of
uniform random variables and
%corresponding
Legendre polynomials, the ratio of the constants in~(\ref{eq:AS-ineq}) is
close to one and hence the convergence of conjugate gradients is not as
pessimistic as predicted by the bound.
%We note that this theoretical bound
%does not apply to the variant of the preconditioner$~M_{\mathrm{{gPC}}}$
%introduced in Section~\ref{sec:numerical}.
%Noting that, for a
%fixed degree$~\ell$ of the polynomial expansion, increasing the size of the
%stochastic dimension$~N$\ does not change the values of eigenvalues in the
%matrix$~A_{\ell}$, the equivalence~(\ref{eq:AS-ineq}) is independent
%of$~N$\ and so is the rate of convergence predicted by Theorem~\ref{thm:cond}.
%On the other hand, it appears that increasing the degree of polynomial
%expansion leads to multiplicative growth of the condition number. However, our
%numerical experiments with uniform random variables and the corresponding
%Legendre polynomials indicate that the bound is equal to one.
%might be quite tight and the convergence is still remarkably fast.

\end{remark}

In the next section, we discuss several modifications of the method and the preconditioner.

\section{Variants and implementation remarks}

\label{sec:variants}Clearly, there are many other ways of setting up a
hierarchical preconditioner. These possibilities follow by considering the
block inverse~(\ref{eq:block-inverse}) and writing it in a more general form,
which can be subsequently used in the approximation of the preconditioner from
Algorithm~\ref{alg:hierarchical-schur-prec}, as%
\begin{equation}
M=\left[
\begin{array}
[c]{cc}%
I_{A} & 0\\
-M_{D}^{3}C & I_{D}%
\end{array}
\right]  \left[
\begin{array}
[c]{cc}%
M_{S} & 0\\
0 & M_{D}^{2}%
\end{array}
\right]  \left[
\begin{array}
[c]{cc}%
I_{A} & -BM_{D}^{1}\\
0 & I_{D}%
\end{array}
\right]  , \label{eq:general-prec}%
\end{equation}
so that $M_{D}^{i}$\ , $i=1,2,3$, approximate $D^{-1}$\ and$~M_{S}$
approximates$~S^{-1}$. Our main approximation in
Algorithm~\ref{alg:hierarchical-schur-prec} is in using the hierarchy of
matrices $A_{\ell}$, $\ell=P-1,\dots,0$ in place of$~M_{S}$\ on each level.
Next, in our case$~D$ is block-diagonal. Thus computing its inverse means
solving independently a number of systems, where each one of them has the same
size (and sparsity structure) as the deterministic problem for the mean. In
fact, the diagonal blocks are just scalar multiples of the \textquotedblleft
mean-value\textquotedblright\ matrix$~K_{0}(=A_{0})$. In our implementation,
we have replaced the exact solves of$~D$ by independent loops of
preconditioned Krylov subspace iterations for each diagonal block of$~D$ using
the mean-value preconditioner$~M_{0}$. In the numerical experiments we have
tested convergence with the following choices of$~M_{0}$: no preconditioner,
simple diagonal preconditioner, and the exact LU\ decomposition of the
block$~A_{0}$ (which converges in one iteration). So this variant of the
hierarchical Schur complement preconditioner involves multiple loops of inner
iterations and thus possibly changes in every outer iteration. In order to
accommodate such variable preconditioner, it is generally\ recommended to use
a flexible Krylov subspace method such as flexible CG~\cite{Notay-2000-FCG},
FGMRES~\cite{Saad-1993-FIP}, or GMRESR~\cite{vanderVorst-1994-GFN}.
Nevertheless, we have observed essentially the same convergence in terms of
outer iterations with both variants of the conjugate gradients, the flexible
and the standard one as well. The convergence seems also to be independent of
the choice of$~M_{0}$ and in this contribution we do not advocate any specific
choice. Next, one can in general replace the action of any$~M_{D}^{i}$,
$i=1,2,3$, by the action of just$~M_{0}$ itself. However it is well-known from
iterative substructuring cf., e.g.,~\cite[Section 4.4]{Smith-1996-DD}, that
even if $M_{D}$ is spectrally equivalent to$~D^{-1}$, the resulting
preconditioner might not be spectrally equivalent to the original problem.
%In agreement with this observation, replacing all three$~M_{D}^{i}$ operators
%by$~M_{0}$ in Algorithm~\ref{alg:hierarchical-schur-prec} did not lead to
%satisfactory convergence results (worse than the mean-based preconditioner).
%Finally, we remark that a number of closely
%related %approximate
%preconditioners %, obtained, e.g., by setting$~B$%
%\ in$~$(\ref{eq:general-prec}) to zero,
%has been studied by Tipireddy and the
%two authors~\cite{Tipireddy-2011-CSM}.

It also appears that one can modify not only the preconditioner, but also the
set up of the method itself. Namely, inspired by the iterative substructuring
cf., e.g.,~\cite{Toselli-2005-DDM}, one can reduce the system given by$~A_{P}$
to the system given by the Schur complement~$S_{P-1}$ used subsequently in the
iterations. So, in the first step, cf.~(\ref{eq:global-matrix-hierarchy-2}),
we eliminate $u_{P}^{P}$ and define $u_{P-1}\equiv u_{P}^{P-1}$, which yields
\begin{equation}
S_{P-1}u_{P-1}=g_{P-1}, \label{eq:global-Schur-system}%
\end{equation}
where%
\[
S_{P-1}=A_{P-1}-B_{P}D_{P}^{-1}C_{P},\quad\text{and}\quad g_{P-1}=f_{P}%
^{P-1}-B_{P}D_{P}^{-1}f_{P}^{P}.
\]
After convergence, the variables $u_{P}^{P}$ are recovered from
\[
u_{P}^{P}=D_{P}^{-1}\left(  f_{P}^{P}-C_{P}u_{P-1}\right)  .
\]
There are two advantages of the a-priori elimination of the second block:
first, because the system (\ref{eq:global-Schur-system}) will be solved
iteratively, the iterations can be performed on a much smaller system and
also, at least for symmetric, positive definite problems, the condition number
of the Schur complement cannot be higher than the one of the original
problem~\cite{Smith-1996-DD} even if one uses a diagonal
preconditioning~\cite{Mandel-1990-BDS}. The preconditioner~$M_{P-1}$ for the
system~(\ref{eq:global-Schur-system}) is then the same as in
Algorithm~\ref{alg:hierarchical-schur-prec} except that the for-loops are
performed only for all levels$~\ell=1,\dots,P-1$. However, this reduction is
theoretically justified only when exact solves for the block diagonal
matrix~$D_{P}$\ are available. In general, if one uses only approximate
solves, e.g., by performing inner/outer Krylov iterations for $D_{P}$
and\ $S_{P-1}$\ respectively, the global system matrix becomes variable as
well, this might lead to the loss of orthogonality and poor performance of the
method. Our numerical experiments indicated that the preconditioned iterations
for $A_{P}$ and $S_{P-1}$ perform identically, but we do not advocate to use
a-priori reduction to the Schur complement in general.

\section{Numerical examples}

\label{sec:numerical}
%In the first set of experiments,
We have implemented the stochastic Galerkin finite element method for the
model elliptic problem~(\ref{eq:model-1})-(\ref{eq:model-2}) on a square
domain $\left[  0,1\right]  \times\left[  0,1\right]  $ uniformly discretized
by $10\times10$ Lagrangean bilinear finite elements. The mean value of the
coefficient$~k$ was set to $k_{0}=1$. The coefficients in the covariance
function~$C$ defined by eq.~(\ref{eq:cov}) were set to $L=0.5$ and
$\sigma=0.5$, so the coefficient of variation is given as $CoV=\sigma
/k_{0}=\sigma=50\%$. The $15$ dominant eigenvalues of the discretized
eigenvalue problem~(\ref{eq:int-eigproblem}) are shown in
Figure~\ref{fig:eigs}. We have studied convergence of the flexible version of
the conjugate gradient method (FCG) without a preconditioner, with a global
mean-based preconditioner$~M_{\mathrm{m}}$ by Powell and
Elman~\cite{Powell-2009-BDP},
%$M_{m}=I\otimes A_{0}^{LU}$ where
%\thinspace$A_{0}^{LU}$\ is the LU-decomposition of the \textquotedblleft
%mean-value\textquotedblright\ block$~A_{0}$,
with the block symmetric Gauss-Seidel preconditioner~$M_{\mathrm{bGS}}$ (with
zero initial guess) and with the hierarchical Schur complement
preconditioner$~M_{\mathrm{HS}}$. The convergence results
%in terms of iterations and condition number estimates
%from the L\'{a}nczos sequence in conjugate gradients
are summarized in Tables~\ref{tab:Legendre-var-N}-\ref{tab:Legendre-var-h}. We
have observed essentially the same convergence of the standard conjugate
gradients compared to the flexible\ version, which is reported in the tables.
Also, in our experience, the convergence rates were independent of the choice
of the mean-value preconditioner$~M_{0}$ (no preconditioner, diagonal
preconditioner and the LU-decomposition of the \textquotedblleft
mean-value\textquotedblright\ block$~A_{0}$) used in inner iterations of the
preconditioner for the diagonal block solves with the same relative residual
tolerance as in the outer iterations. From Tables~\ref{tab:Legendre-var-N}
and~\ref{tab:Legendre-var-P} it appears that the convergence depends only
mildly on the stochastic dimension$~N$ and the order of polynomial
expansion$~P$, respectively.
%similarly as predicted by Theorem~\ref{thm:cond}.
Table~\ref{tab:Legendre-var-sigma} indicates a modest dependence on the value
of the standard deviation $\sigma$, and finally Table~\ref{tab:Legendre-var-h}
indicates that the convergence is independent of the mesh size$~h$.
%coefficient of variation $CoV=\sigma/k_{0}$.
We note that for $CoV>55\%$ the problem is no longer guaranteed to be
elliptic, and the global matrix$~A$ is not positive definite.
%It should be also
%noted that no preconditioning, simple diagonal preconditioning and use of
%$A_{0}^{LU}$ in the inner iterations lead to the same count of outer
%iterations.

%figure eigenvalues
\begin{figure}[ptb]
\begin{center}
\includegraphics[width=12.5cm]{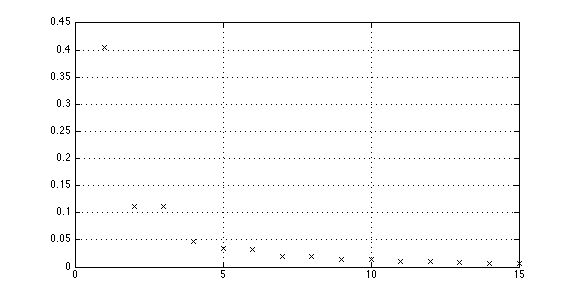}
\end{center}
\caption{The $15$ dominant eigenvalues with the covariance kernel
(\ref{eq:cov}) (in this plot $\sigma=1)$.}%
\label{fig:eigs}%
\end{figure}

%table Legendre 1
\begin{table}[pth]
\caption{Convergence of (flexible) conjugate gradients for the full system
matrix $A$, for $A$ preconditioned by the mean-based preconditioner
$M_{\mathrm{m}}$, by the block Gauss-Seidel preconditioner $M_{\mathrm{bGS}}$,
and by the hierarchical Schur complement preconditioner $M_{\mathrm{HS}}$. The
coefficient of variation of the uniform random field is $CoV=50\%$, polynomial
degree is $P=4$, and the stochastic dimension $N$ is variable. Here, $ndof$ is
the dimension of $A$, $iter$ is the number of iterations with the relative
residual tolerance~$10^{-8}$, and $\kappa$ is the condition number estimate
from the L\'{a}nczos sequence in conjugate gradients.}%
\label{tab:Legendre-var-N}
\centering
\begin{tabular}
[c]{|cc|cc|cc|cc|cc|}\hline
\multicolumn{2}{|c|}{setup} & \multicolumn{2}{|c|}{$A$} &
\multicolumn{2}{|c|}{$M_{\mathrm{m}}A$} &
\multicolumn{2}{|c|}{$M_{\mathrm{bGS}}A$} &
\multicolumn{2}{|c|}{$M_{\mathrm{HS}}A$}\\\hline
$N$ & $ndof$ & $iter$ & $\kappa$ & $iter$ & $\kappa$ & $iter$ & $\kappa$ &
$iter$ & $\kappa$\\\hline
1 & 605 & 173 & 1965.4 & 12 & 2.0127 & 5 & 1.0507 & 5 & 1.0465\\
2 & 1815 & 531 & 5333.3 & 15 & 2.7340 & 6 & 1.1279 & 6 & 1.1236\\
3 & 4235 & 745 & 9876.9 & 16 & 2.9995 & 7 & 1.1693 & 6 & 1.1514\\
4 & 8470 & 902 & 17,150.2 & 17 & 3.3413 & 7 & 1.2131 & 7 & 1.2028\\
5 & 15,246 & 1033 & 17,275.8 & 18 & 3.5891 & 7 & 1.2447 & 7 & 1.2434\\
6 & 25,410 & 1037 & 17,333.5 & 18 & 3.6349 & 7 & 1.2501 & 7 & 1.2559\\
7 & 39,930 & 1040 & 17,348.9 & 19 & 4.0993 & 8 & 1.3202 & 7 & 1.3146\\
8 & 59,895 & 1081 & 17,360.6 & 19 & 4.0597 & 8 & 1.3198 & 7 & 1.3182\\\hline
\end{tabular}
\end{table}

%table Legendre 2
\begin{table}[pth]
\caption{Convergence of (flexible) conjugate gradients for the full system
matrix $A$, for $A$ preconditioned by the mean-based preconditioner
$M_{\mathrm{m}}$, by the block Gauss-Seidel preconditioner $M_{\mathrm{bGS}}$,
and by the hierarchical Schur complement preconditioner $M_{\mathrm{HS}}$. The
stochastic dimension is $N=4$, $CoV=50\%$, and the polynomial degree $P$ is
variable. The other headings are same as in Table~\ref{tab:Legendre-var-N}.}%
\label{tab:Legendre-var-P}%
\centering
\begin{tabular}
[c]{|cc|cc|cc|cc|cc|}\hline
\multicolumn{2}{|c|}{setup} & \multicolumn{2}{|c|}{$A$} &
\multicolumn{2}{|c|}{$M_{\mathrm{m}}A$} &
\multicolumn{2}{|c|}{$M_{\mathrm{bGS}}A$} &
\multicolumn{2}{|c|}{$M_{\mathrm{{HS}}}A$}\\\hline
$P$ & $ndof$ & $iter$ & $\kappa$ & $iter$ & $\kappa$ & $iter$ & $\kappa$ &
$iter$ & $\kappa$\\\hline
1 & 605 & 134 & 625.6 & 9 & 1.6391 & 5 & 1.0626 & 5 & 1.0624\\
2 & 1815 & 315 & 1903.2 & 13 & 2.2379 & 6 & 1.1117 & 6 & 1.1109\\
3 & 4235 & 586 & 5721.1 & 15 & 2.8122 & 7 & 1.1658 & 6 & 1.1559\\
4 & 8470 & 902 & 17,150.2 & 17 & 3.3413 & 7 & 1.2131 & 7 & 1.2028\\
5 & 15,246 & 1402 & 29,751.0 & 18 & 3.7824 & 7 & 1.2538 & 7 & 1.2426\\
6 & 25,410 & 1943 & 49,842.4 & 19 & 4.1534 & 8 & 1.2921 & 7 & 1.2798\\
7 & 39,930 & 2568 & 83,056.6 & 20 & 4.4708 & 8 & 1.3219 & 7 & 1.3125\\
8 & 59,895 & 3267 & 136,419.0 & 20 & 4.7371 & 8 & 1.3472 & 7 & 1.3398\\\hline
\end{tabular}
\end{table}

%table Legendre 3
\begin{table}[pth]
\caption{Convergence of (flexible) conjugate gradients for the full system
matrix $A$, its first Schur complement $S$, for $A$ preconditioned by the
global mean-based preconditioner $M_{\mathrm{m}}$, by the block Gauss-Seidel
preconditioner $M_{\mathrm{bGS}}$, and by the hierarchical Schur complement
preconditioner $M_{\mathrm{HS}}$. Here, the size of $A$ is $8470$ $ndof$, the
stochastic dimension is $N=4$, the polynomial degree is $P=4$, the mean is
$k_{0}=1$, and the coefficient of variation $CoV$ is variable. The other
headings are same as in Table~\ref{tab:Legendre-var-N}.}%
\label{tab:Legendre-var-sigma}%
\centering
\begin{tabular}
[c]{|c|cc|cc|cc|cc|}\hline
\multicolumn{1}{|c|}{setup} & \multicolumn{2}{|c|}{$A$} &
\multicolumn{2}{|c|}{$M_{\mathrm{m}}A$} &
\multicolumn{2}{|c|}{$M_{\mathrm{bGS}}A$} &
\multicolumn{2}{|c|}{$M_{\mathrm{{HS}}}A$}\\\hline
$CoV(\%)$ & $iter$ & $\kappa$ & $iter$ & $\kappa$ & $iter$ & $\kappa$ & $iter$
& $\kappa$\\\hline
5 & 694 & 15,556.3 & 6 & 1.0960 & 3 & 1.0008 & 3 & 1.0009\\
15 & 739 & 15,673.2 & 9 & 1.3514 & 4 & 1.0090 & 4 & 1.0089\\
25 & 804 & 15,912.5 & 11 & 1.7021 & 5 & 1.0314 & 5 & 1.0304\\
35 & 833 & 16,286.1 & 13 & 2.1808 & 6 & 1.0770 & 5 & 1.0664\\
45 & 877 & 16,815.9 & 16 & 2.8773 & 6 & 1.1510 & 6 & 1.1414\\
55 & 926 & 17,539.6 & 19 & 3.9523 & 8 & 1.2948 & 7 & 1.2830\\\hline
\end{tabular}
\end{table}

%table Legendre 4
\begin{table}[pth]
\caption{Convergence of (flexible) conjugate gradients for the full system
matrix $A$, for $A$ preconditioned by the global mean-based preconditioner
$M_{\mathrm{m}}$, by the block Gauss-Seidel preconditioner $M_{\mathrm{bGS}}$,
and by the hierarchical Schur complement preconditioner $M_{\mathrm{HS}}$.
Here, the stochastic dimension is $N=4$, the polynomial degree is $P=4$, the
mean is $k_{0}=1$, the coefficient of variation is $CoV=50\%$, and the size of
the finite element mesh~$h$ is variable. The other headings are same as in
Table~\ref{tab:Legendre-var-N}.}%
\label{tab:Legendre-var-h}%
\centering
\begin{tabular}
[c]{|cc|cc|cc|cc|cc|}\hline
\multicolumn{2}{|c|}{setup} & \multicolumn{2}{|c|}{$A$} &
\multicolumn{2}{|c|}{$M_{\mathrm{m}}A$} &
\multicolumn{2}{|c|}{$M_{\mathrm{bGS}}A$} &
\multicolumn{2}{|c|}{$M_{\mathrm{{HS}}}A$}\\\hline
$h$ & $ndof$ & $iter$ & $\kappa$ & $iter$ & $\kappa$ & $iter$ & $\kappa$ &
$iter$ & $\kappa$\\\hline
$1/5$ & 2520 & 404 & 4847.5 & 16 & 3.2484 & 7 & 1.2022 & 6 & 1.1790\\
$1/10$ & 8470 & 902 & 17,150.2 & 17 & 3.3413 & 7 & 1.2131 & 7 & 1.2028\\
$1/15$ & 17,920 & 1386 & 36,716.6 & 17 & 3.3145 & 7 & 1.2063 & 7 & 1.2047\\
$1/20$ & 30,870 & 1883 & 63,535.2 & 17 & 3.3463 & 7 & 1.2110 & 7 & 1.2032\\
$1/25$ & 47,320 & 2383 & 97,605.6 & 17 & 3.3473 & 7 & 1.2112 & 7 & 1.2032\\
$1/30$ & 67,270 & 2872 & 138,929.0 & 17 & 3.3190 & 7 & 1.2070 & 7 &
1.2054\\\hline
\end{tabular}
\end{table}

%table Legendre 5 - block count
\begin{table}[pth]
\caption{Numbers of blocks in the full system matrix and the \textquotedblleft
work-count\textquotedblright\ in the application of the preconditioner~$M$,
when one of the parameters $N$ or $P$ is changing and the other one is set to
$4$, cf.~Figure~\ref{fig:NP}. Here $n_{b}$ is the total number of blocks,
$n_{db}$ is the number of diagonal blocks, which is the same as the number of
solves in the application of the mean-based preconditioner $M_{m}$, $n_{m}$ is
the number of block matrix-vector multiplications in the action of the
preconditioner $M$, and $n_{ds}$ is the number of its block diagonal solves.}%
\label{tab:Legendre-work}
\centering
\begin{tabular}
[c]{|c|cc|cc|}\hline
$N$ or $P$ & $n_{b}$ & $n_{db}$ & $n_{m}$ & $n_{ds}$\\\hline
1 & 13 & 5 & 8 & 9\\
2 & 55 & 15 & 40 & 29\\
3 & 155 & 35 & 120 & 69\\
4 & 350 & 70 & 280 & 139\\
5 & 686 & 126 & 560 & 251\\
6 & 1218 & 210 & 1008 & 419\\
7 & 2010 & 330 & 1680 & 659\\
8 & 3135 & 495 & 2640 & 989\\\hline
\end{tabular}
\end{table}

Table~\ref{tab:Legendre-work} summarizes the block count in the structure of
the global Galerkin matrix$~A$ obtained using the KL\ expansion, cf.
Figure~\ref{fig:NP}, when either of the parameters $N$\ or $P$ changes and the
other one is set to be equal to four. The two choices lead to slightly
different block sparsity structures of $A$, however the numbers of blocks are
the same. Let us denote by $n_{b}$ the total number of blocks in$~A$ and by
$n_{db}$ the number of its diagonal blocks. Note that one application of the
mean-based preconditioner requires $n_{db}$ solves of the diagonal blocks. The
columns three and four in Table~\ref{tab:Legendre-work} contain the numbers of
block matrix-vector multiplications$~n_{m}$ and block diagonal solves$~n_{ds}$
performed in one action of the hierarchical Schur preconditioner. From
Algorithm~\ref{alg:hierarchical-schur-prec} we obtain that
\[
n_{m}=n_{b}-n_{db},
\]
where half of multiplications is performed in the first for-loop and the other
half in the second, and%
\[
n_{ds}=2(n_{db}-1)+1,
\]
which follows from the two for-loops and one solve of the first block$~A_{0}$.
Hence one action of the hierarchical Schur preconditioner requires nearly the
same number of computations as one global Galerkin matrix-vector
multiplications, $n_{m}\approx n_{b}$, and two applications of the mean-based
preconditioner, $n_{ds}\approx2n_{db}$. It is important to note that whereas
the application of the mean-based preconditioner can be performed fully in
parallel, the two for-loops in Algorithm~\ref{alg:hierarchical-schur-prec} are
sequential, and thus the eventual parallelization can be performed only within
each step of these for-loops. The work count of ~$M_{\mathrm{bGS}}$, which is
block sequential, is given by $2n_{db}$ diagonal solves, and $1.5$ (or $2$, if
the initial guess of GS is nonzero)\ times of block matrix-vector
multiplications compared to~$M_{\mathrm{HS}}$.

In the second set of experiments, we have tested convergence of the
preconditioner with the same physical domain and parameter setting, except
assuming that the random coefficient$~k$ has lognormal distribution with the
coefficient of variation being set to $CoV=\sigma_{\log}/\mu_{\log}=100\%$. We
note that in order to guarantee existence and uniqueness of the solution, we
have used twice the order of polynomial expansion of the coefficient$~k$ than
of the solution, cf.~\cite{Matthies-2005-GML}. Such discretization is done
within the gPC framework, see Remark~\ref{rem:K-L}, using Hermite
polynomials~\cite{Ghanem-1999-NGS}, and leads to a fully block dense structure
of the global Galerkin matrix$~A$.
%as illustrated by Figure~\ref{fig:N4P4log_mat}.
Therefore the solves involving submatrices$~D_{\ell}$, $\ell=1,\dots,P$, in
the pre- and post-correction steps are no longer block diagonal. Our numerical
tests using both, direct\ and iterative solves with the$~D_{\ell}$, and using
the same tolerance as for the outer iterations, lead to the same count of
outer iterations. The performance results are summarized in
Tables~\ref{tab:Lognormal-var-N}-\ref{tab:Lognormal-var-h}. The convergence
rate reported in Table~\ref{tab:Lognormal-var-N} indicates a mild dependence
on the stochastic dimension$~N$, Table~\ref{tab:Lognormal-var-P} indicates a
modest dependence on the order of the polynomial expansion$~P$, and
Table~\ref{tab:Lognormal-var-CoV} indicates also a modest dependence on the
coefficient of variation$~CoV$. From Table~\ref{tab:Lognormal-var-h} we see
that the convergence is nearly independent of the mesh size$~h$.
%We note that the theoretical bound of Theorem~\ref{thm:cond} would not apply in
%establishing the convergence rate.
%When either of the parameters $N$\ or $P$
%changes and the other one is set to be equal to four, the (block) sizes of the
%global Galerkin matrix are the same, however unlike for the KL expansion we
%obtain different numbers of blocks in the two cases. Importantly, one action
%of the hierarchical Schur preconditioner$~M_{\mathrm{{gPC}}}$ requires
%relatively less computations than$~M$: the block matrix-vector multiplications
%are reduced by neglecting all blocks of$~D_{\ell}$, $\ell=1,\dots P$, and thus
%$n_{m}<n_{b}$, but still we use the solves obtained by using all of the
%diagonal blocks of $D_{\ell}$, $\ell=1,\dots P$, and thus $n_{ds}%
%\approx2n_{db}$.
The performance of both preconditioners~$M_{\mathrm{bGS}}$ and~$M_{\mathrm{HS}%
}$\ is significantlly better compared to the mean-based
preconditioner~$M_{\mathrm{m}}$. Also, we see that~$M_{\mathrm{HS}}$ performs
a bit better than~$M_{\mathrm{bGS}}$. However, we must note that
~$M_{\mathrm{HS}}$ is also more computationally intensive because it requires
solves with larger diagonal submatrices$~D_{\ell}$, for all levels
$\ell=1,\dots,P$, and a work count comparison with~$M_{\mathrm{bGS}}$ is not
straightforward. As before, the two for-loops corresponding to
Algorithm~\ref{alg:hierarchical-schur-prec} are sequential, and thus the
eventual paralelisation can be performed only within each step in the for-loop.

%\begin{figure}[th]
%\centering
%\subfigure[Structure of the matrix gives $3090$ blocks]{
%\includegraphics[width=6.5cm]{graphics/N4P4log_mat.pdf}
%\label{fig:N4P4log_mat}
%} \subfigure[Structure of the preconditioner gives $2040$ blocks]{
%\includegraphics[width=6.5cm]{graphics/N4P4log_prec.pdf}
%\label{fig:N4P4log_prec}
%}\caption{Hierarchical structure of the global stochastic Galerkin matrix and
%of the preconditioner$~M_{\mathrm{{gPC}}}$ corresponding to the gPC expansion
%of the random parameter $k$ with lognormal distribution, using $N=4$
%stochastic dimensions and $P=4$ as the order of the polynomial expansion. The
%sub-blocks correspond to the polynomials of orders $P=1,2,3$. }%
%\label{fig:NP_Lognormal}%
%\end{figure}

%table Lognormal 1
\begin{table}[pth]
\caption{Convergence of (flexible) conjugate gradients for the full system
matrix~$A$ obtained by the gPC expansion of the lognormal field, for~$A$
preconditioned by the mean-based preconditioner $M_{\mathrm{m}}$, by the block
Gauss-Seidel preconditioner $M_{\mathrm{bGS}}$, and by the hierarchical Schur
complement preconditioner $M_{\mathrm{HS}}$. Polynomial degree is fixed to
$P=4$, the coefficient of variation of the lognormal random field is
$CoV=100\%$, and the stochastic dimension $N$ is variable. The other headings
are same as in Table~\ref{tab:Legendre-var-N}.}%
\label{tab:Lognormal-var-N}
\centering
\begin{tabular}
[c]{|cc|cc|cc|cc|cc|}\hline
\multicolumn{2}{|c|}{setup} & \multicolumn{2}{|c|}{$A$} &
\multicolumn{2}{|c|}{$M_{\mathrm{m}}A$} & \multicolumn{2}{|c|}{$M_{
\mathrm{bGS}}A$} & \multicolumn{2}{|c|}{$M_{\mathrm{{HS}}}A$}\\\hline
$N$ & $ndof$ & $iter$ & $\kappa$ & $iter$ & $\kappa$ & $iter$ & $\kappa$ &
$iter$ & $\kappa$\\\hline
1 & 605 & 585 & 51,376.4 & 48 & 28.7589 & 15 & 3.4192 & 15 & 3.4000\\
2 & 1815 & 1396 & 58,718.8 & 61 & 37.1593 & 17 & 3.7490 & 16 & 3.6244\\
3 & 4235 & 1770 & 69,054.8 & 62 & 38.0715 & 17 & 3.7380 & 16 & 3.7632\\
4 & 8470 & 2016 & 70,143.6 & 66 & 43.6525 & 19 & 4.2935 & 16 & 4.1669\\\hline
\end{tabular}
\end{table}

%table Lognormal 2
\begin{table}[pth]
\caption{Convergence of (flexible) conjugate gradients for the full system
matrix~$A$ obtained by the gPC expansion of the lognormal field, for~$A$
preconditioned by the mean-based preconditioner $M_{\mathrm{m}}$, by the block
Gauss-Seidel preconditioner $M_{\mathrm{bGS}}$, and by the hierarchical Schur
complement preconditioner $M_{\mathrm{HS}}$. Stochastic dimension is fixed to
$N=4$, the coefficient of variation of the lognormal random field is
$CoV=100\%$, and the polynomial degree $P$ is variable. The other headings are
same as in Table~\ref{tab:Legendre-var-N}.}%
\label{tab:Lognormal-var-P}
\centering
\begin{tabular}
[c]{|cc|cc|cc|cc|cc|}\hline
\multicolumn{2}{|c|}{setup} & \multicolumn{2}{|c|}{$A$} &
\multicolumn{2}{|c|}{$M_{\mathrm{m}}A$} &
\multicolumn{2}{|c|}{$M_{\mathrm{bGS}}A$} &
\multicolumn{2}{|c|}{$M_{\mathrm{{HS}}}A$}\\\hline
$P$ & $ndof$ & $iter$ & $\kappa$ & $iter$ & $\kappa$ & $iter$ & $\kappa$ &
$iter$ & $\kappa$\\\hline
1 & 605 & 134 & 578.2 & 15 & 3.4954 & 8 & 1.3910 & 7 & 1.3856\\
2 & 1815 & 329 & 2027.3 & 28 & 8.9450 & 12 & 1.9742 & 10 & 1.9289\\
3 & 4235 & 804 & 10,048.4 & 44 & 20.0366 & 15 & 2.8670 & 13 & 2.7955\\
4 & 8470 & 2016 & 70,143.6 & 66 & 43.6525 & 19 & 4.2935 & 16 & 4.1669\\\hline
\end{tabular}
\end{table}

%table Lognormal 3
\begin{table}[pth]
\caption{Convergence of (flexible) conjugate gradients for the full system
matrix~$A$ obtained by the gPC expansion of the lognormal field, for~$A$
preconditioned by the mean-based preconditioner $M_{\mathrm{m}}$, by the block
Gauss-Seidel preconditioner $M_{\mathrm{bGS}}$, and by the hierarchical Schur
complement preconditioner $M_{\mathrm{HS}}$. Here, the size of $A$ is $8470$
$ndof$, the stochastic dimension is $N=4$, the polynomial degree is $P=4$, and
the coefficient of variation of the lognormal field~$CoV$ is variable. The
other headings are same as in Table~\ref{tab:Legendre-var-N}.}%
\label{tab:Lognormal-var-CoV}
\centering
\begin{tabular}
[c]{|c|cc|cc|cc|cc|}\hline
\multicolumn{1}{|c|}{setup} & \multicolumn{2}{|c|}{$A$} &
\multicolumn{2}{|c|}{$M_{\mathrm{m}}A$} &
\multicolumn{2}{|c|}{$M_{\mathrm{bGS}}A$} &
\multicolumn{2}{|c|}{$M_{\mathrm{{HS}}}A$}\\\hline
$CoV\,(\%)$ & $iter$ & $\kappa$ & $iter$ & $\kappa$ & $iter$ & $\kappa$ &
$iter$ & $\kappa$\\\hline
25 & 719 & 7378.4 & 16 & 3.2356 & 7 & 1.1761 & 7 & 1.1776\\
50 & 1039 & 16,014.8 & 29 & 9.3553 & 11 & 1.7685 & 10 & 1.7836\\
75 & 1511 & 35,317.3 & 46 & 22.2147 & 15 & 2.8198 & 13 & 2.8454\\
100 & 2016 & 70,143.6 & 66 & 43.6525 & 19 & 4.2935 & 16 & 4.1669\\
125 & 2591 & 116,678.0 & 85 & 72.7584 & 23 & 5.9776 & 19 & 5.5362\\
150 & 3209 & 178,890.0 & 103 & 107.0670 & 26 & 7.7459 & 21 & 6.8507\\\hline
\end{tabular}
\end{table}

%table lognormal 4
\begin{table}[pth]
\caption{Convergence of (flexible) conjugate gradients for the full system
matrix~$A$ obtained by the gPC expansion of the lognormal field, for~$A$
preconditioned by the mean-based preconditioner $M_{\mathrm{m}}$, by the block
Gauss-Seidel preconditioner $M_{\mathrm{bGS}}$, and by the hierarchical Schur
complement preconditioner $M_{\mathrm{HS}}$. Here, the stochastic dimension is
$N=4$, the polynomial degree is $P=4$, the coefficient of variation of the
lognormal random field is $CoV=100\%$, and the size of the finite element
mesh~$h$ is variable. The other headings are same as in
Table~\ref{tab:Legendre-var-N}.}%
\label{tab:Lognormal-var-h}
\centering
\begin{tabular}
[c]{|cc|cc|cc|cc|cc|}\hline
\multicolumn{2}{|c|}{setup} & \multicolumn{2}{|c|}{$A$} &
\multicolumn{2}{|c|}{$M_{\mathrm{m}}A$} &
\multicolumn{2}{|c|}{$M_{\mathrm{bGS}}A$} &
\multicolumn{2}{|c|}{$M_{\mathrm{{HS}}}A$}\\\hline
$h$ & $ndof$ & $iter$ & $\kappa$ & $iter$ & $\kappa$ & $iter$ & $\kappa$ &
$iter$ & $\kappa$\\\hline
$1/5$ & 2520 & 831 & 17,695.3 & 59 & 40.6232 & 18 & 3.9885 & 15 & 3.8361\\
$1/10$ & 8470 & 2016 & 70,143.6 & 66 & 43.6525 & 19 & 4.2935 & 16 & 4.1669\\
$1/15$ & 17,920 & 3377 & 158,334.0 & 68 & 44.4170 & 19 & 4.3764 & 16 &
4.2394\\
$1/20$ & 30,870 & 4395 & 275,686.0 & 69 & 44.8882 & 19 & 4.3742 & 17 &
4.2510\\
$1/25$ & 47,320 & 5600 & 429,551.0 & 69 & 44.9413 & 20 & 4.3986 & 17 &
4.2592\\
$1/30$ & 67,270 & 7180 & 626,475.0 & 71 & 45.1100 & 19 & 4.3732 & 17 &
4.2630\\\hline
\end{tabular}
\end{table}

%%table Lognormal 5 - block count
%\begin{table}[pth]
%\caption{Numbers of blocks in the full system matrix and the ``work-count'' in
%the application of the preconditioner~$M_{\mathrm{{gPC}}}$, when the
%stochastic dimension $N$ changes and the polynomial order $P=4$ (left), and
%when $P$ changes and $N=4$ (right), cf.~Figure~\ref{fig:NP_Lognormal}. Here
%$n_{b}$ is the total number of blocks, $n_{db}$ is the number of diagonal
%blocks, which is the same as the number of solves in the application of the
%mean-based preconditioner $M_{m}$, $n_{m}$ is the number of block
%matrix-vector multiplications in the action of the preconditioner
%$M_{\mathrm{{gPC}}}$, and $n_{ds}$ is the number of its block diagonal
%solves.}%
%\label{tab:Lognormal-work}%
%\centering
%\begin{tabular}
%[c]{|c|cc|cc||c|cc|cc|}\hline
%$N$ & $n_{b}$ & $n_{db}$ & $n_{m}$ & $n_{ds}$ & $P$ & $n_{b}$ & $n_{db}$ &
%$n_{m}$ & $n_{ds}$\\\hline
%1 & 25 & 5 & 20 & 9 & 1 & 13 & 5 & 8 & 9\\
%2 & 193 & 15 & 146 & 29 & 2 & 135 & 15 & 60 & 29\\
%3 & 901 & 35 & 626 & 69 & 3 & 715 & 35 & 500 & 69\\
%4 & 3090 & 70 & 1970 & 139 & 4 & 3090 & 70 & 1970 & 139\\
%5 & 8606 & 126 & 5050 & 251 & 5 & 10,158 & 126 & 7722 & 251\\
%6 & 20,650 & 210 & 11,200 & 419 & 6 & 29,448 & 210 & 21,566 & 419\\
%7 & 44,318 & 330 & 22,316 & 659 & 7 & 73,820 & 330 & 59,238 & 659\\
%8 & 87,231 & 495 & 40,956 & 989 & 8 & 170,505 & 495 & 134,298 & 989\\\hline
%\end{tabular}
%\end{table}

The numerical experiments presented here were implemented using a sequential
code in \texttt{Matlab}, version \texttt{7.12.0.635 (R2011a)}, and therefore
we do not report on computational times.

\section{Conclusion}

\label{sec:summary} We have presented a hierarchical Schur complement
preconditioner for the iterative solution of the systems of linear algebraic
equations obtained from the stochastic Galerkin finite element
discretizations. The preconditioner takes an advantage of the recursive
hierarchical two-by-two structure of the global matrix, with one of the
submatrices block diagonal. We have compared its convergence using (flexible)
conjugate gradients without any preconditioner, with the mean-based
preconditioner which requires one block diagonal solve per iteration, and with
the block version of the well-known symmetric Gauss-Seidel method used as a
preconditioner. The algorithm of our preconditioner consists of a loop of
diagonal block solves and a multiplication by the upper block triangle in the
pre-correction loop, and of another loop of diagonal block solves and a
multiplication by the lower block triangle in the post-correction loop. The
loops are sequential throughout the hierarchy of the global matrix, but the
block solves are independent within each level. We have also succesfully
tested the preconditioner in the case of the random coefficient with lognormal
distribution. However, in this case the algorithm involves solves (either
direct or of preconditioned inner iterations) with larger submatrices than
just the diagonal blocks, and a direct comparison to the symmetric block
Gauss-Seidel preconditioner in terms of work count is not straightforward.

In conclusion, our algorithm appears to be more effective in terms of
iterations and work count compared to the block version of the symmetric
Gauss-Seidel method.
%Moreover, unlike the Gauss-Seidel preconditioner which is block sequential,
%our methods allows for some degree of parallelization. 
Our method also allows for the same degree of parallelism as the Gauss-Seidel method,
since both involve solving the block diagonal matrices~$D_\ell$.
It is important to note
that the discussed preconditioners in general rely only on (block-by-block)
matrix-vector multiplies, and their performance will also depend on the choice
of preconditioner$~M_{0}$ for the solves with the diagonal blocks. Clearly,
one can use such solver for each one of the diagonal blocks that might
introduce another level of parallelism, e.g., similarly as recently proposed
in~\cite{Mandel-2012-ABT,Sousedik-2010-AMB-thesis,Sousedik-2011-AMB}. However
such extensions will be studied elsewhere.

\bibliographystyle{wileyj}
\bibliography{schur_nla.bib}

\end{document}